\documentclass[11pt,a4paper]{article}

\usepackage[T1]{fontenc}
\usepackage{lmodern}
\usepackage{authblk}
\usepackage{amsmath,amsthm,amssymb,mathtools}
\usepackage{enumitem}
\usepackage{tikz}
\usetikzlibrary{arrows.meta,positioning}
\usepackage[margin=1in]{geometry}
\usepackage[font=small,labelfont=bf,labelsep=period,justification=raggedright,singlelinecheck=false,hypcap=false]{caption}
\usepackage[colorlinks=true,linkcolor=blue,citecolor=blue,urlcolor=blue]{hyperref}

\theoremstyle{plain}
\newtheorem{theor}{Theorem}[section]
\newtheorem{lema}[theor]{Lemma}
\newtheorem{prop}[theor]{Proposition}
\newtheorem{corollar}[theor]{Corollary}

\theoremstyle{definition}
\newtheorem{defin}[theor]{Definition}
\newtheorem{ex}[theor]{Example}

\theoremstyle{plain}
\newtheorem{note}[theor]{Remark}

\newcommand{\DCPO}{\mathbf{DCPO}}
\newcommand{\Dtop}{\mathbf{DTop}}
\newcommand{\PU}{P_U}
\newcommand{\PL}{P_L}
\newcommand{\PC}{P_C}
\newcommand{\LL}{\mathcal L}
\newcommand{\UU}{\mathcal U}
\newcommand{\CC}{\mathcal C}
\newcommand{\QQ}{\mathcal Q}
\newcommand{\KConv}{K\!\operatorname{Conv}}

\newcommand{\fin}{\mathrm{fin}}
\newcommand{\id}{\mathrm{id}}
\newcommand{\Bot}{\bot}
\newcommand{\Top}{\top}
\newcommand{\sig}{\sigma}
\newcommand{\upar}{\uparrow}
\newcommand{\down}{\downarrow}

\newcommand{\intr}{\operatorname{int}}

\newcommand{\Lim}{\operatorname{Lim}}

\title{Directed Convex Powerspaces and Convex Powerdomains}
\author[1]{Yuxu Chen}
\affil[1]{College of Mathematics, Sichuan University}
\affil[ ]{\texttt{chenyuxu@scu.edu.cn};  
}
\date{}

\begin{document}

\maketitle

\begin{abstract}
It is known that lower powerdomains preserve and reflect both continuity and quasicontinuity, while the preservation of quasicontinuity by upper and convex powerdomains had long been open.  Directed spaces provide a topological extension framework for dcpos. Powerdomains of dcpos can be characterized as the $D$-completions of the corresponding directed powerspaces. Using this observation, the authors proved in 2024 that upper powerdomains do not preserve quasicontinuity. In this paper, we prove that directed convex powerspaces and convex powerdomains do not preserve quasicontinuity, but they do reflect both continuity and quasicontinuity. We also prove that upper powerspaces and upper powerdomains reflect quasicontinuity. Together with the known results for lower and upper powerdomains, these arguments give the complete preservation and reflection profile of the lower, upper and convex powerdomains for continuity and quasicontinuity.
\end{abstract}


\section{Introduction}

Powerdomains are important constructions of domain theory for the denotational semantics of nondeterminism. The classical constructions include Hoare's lower powerdomain, Smyth's upper powerdomain and Plotkin's convex powerdomain~\cite{Plotkin1976,Smyth1978,AbramskyJung1994,GierzEtAl2003,Heckmann1991,Heckmann1991IPL}.  We denote these dcpo powerdomains by $\PL(L)$, $\PU(L)$ and $\PC(L)$, respectively, when the base dcpo is $L$.  
From a categorical point of view, they are free dcpo-algebras of different kinds: $\PL(L)$ is the free inflationary dcpo-semilattice, $\PU(L)$ is the free deflationary dcpo-semilattice, and $\PC(L)$ is the free dcpo-semilattice; see also the general presentation-theoretic treatment of dcpo-algebras in \cite{JungMoshierVickers2008,Koslowski1997,ChenKouLyu2024Free}. In this paper, the term powerdomain refers to the free dcpo-algebras in the category of dcpos.

It is well known that all powerdomains preserve continuity, that is, if $L$ is a continuous dcpo, then so are $\PL(L)$, $\PU(L)$ and $\PC(L)$. Moreover, these constructions admit simple topological representations: $\PL(L)$ is represented by nonempty Scott closed subsets of $L$; $\PU(L)$ by nonempty Scott compact saturated subsets of $L$; and $\PC(L)$ by nonempty compact lenses ordered by the Egli--Milner order if $L$ is coherent. 

In general, for non-continuous dcpos, lower powerdomains still have the simple closed-set representations, but upper and convex powerdomains do not. Therefore, it is known that lower powerdomains preserve and reflect both continuity and quasicontinuity, but whether upper and convex powerdomains preserve and reflect quasicontinuity had been a long-standing problem since the 1990s.  

In 2024, the authors~\cite{ChenKouLyu2024} proved that upper powerdomains do not preserve quasicontinuity by means of directed powerspaces. Directed spaces were introduced by Kou et al.~\cite{LuoXu2017,Xie2021} as a topological extension framework for dcpos, which is equivalent to monotone determined spaces introduced by Ern\'e~\cite{Erne2009}. Since 2020, Kou and Xie et al. have investigated free algebras over directed spaces and given concrete representations of directed lower and upper powerspaces~\cite{Xie2021,ChenKouLyu2024}, which have become powerful tools for studying the corresponding dcpo powerdomains. From the viewpoint of directed spaces, a dcpo is first regarded as its Scott space, a powerdomain construction is replaced by a corresponding directed powerspace construction, and Scott completion then returns the dcpo powerdomain.  In this framework, continuity and quasicontinuity are preserved and reflected by Scott completion for the directed powerspaces considered here.  Consequently, whether a powerdomain is continuous or quasicontinuous can be studied equivalently by asking whether the corresponding directed powerspace has the same property. Compared to powerdomains, directed powerspaces often have a simpler and more transparent finite-generator structure, which makes approximation properties easier to study directly at the directed-space level. 

In 2022, the authors gave simple concrete representations of directed upper powerspaces, denoted by $\UU(X)$, for a directed space $X$.  Later, in 2024, we constructed a quasicontinuous dcpo $L$ such that $\UU(\Sigma L)$ is not quasicontinuous, where $\UU(\Sigma L)$ is the directed upper powerspace of the Scott space $\Sigma L$ of $L$.  The Scott completion of $\UU(\Sigma L)$ is $\PU(L)$, and the failure of quasicontinuity transfers from the directed upper powerspace to the upper powerdomain. The proof relies on the finding that for a quasicontinuous base dcpo $L$, the following conditions are equivalent: $\PU(L)$ is quasicontinuous; $\PU(L)$ is continuous; and $\PU(L)$ is represented by the nonempty Scott compact saturated subsets of $L$. However, for convex powerdomains, the underlying set and topology are more complicated, and a similar property does not hold. Thus it is difficult to use the same method directly to study the quasicontinuity of convex powerdomains.  Nevertheless, the framework of directed powerspaces remains powerful for studying convex powerdomains.  

In this paper, we first give a concrete representation of the directed convex powerspace for a directed space $X$, denoted by $\CC(X)$.  We prove its universal property as the free directed semilattice, and then use it as the main object of study.  We then prove a retraction theorem: if a directed space $X$ has a greatest element, then the directed upper powerspace $\UU(X)$ is a retract of $\CC(X)$.  More explicitly, we construct continuous maps
\[
        e\colon \UU(X)\to \CC(X),\qquad r\colon \CC(X)\to \UU(X)
\]
with $r\circ e=\id_{\UU(X)}$. 
By this result, we use the known upper-powerdomain counterexample to show that there exists a quasicontinuous dcpo $L$ such that the directed convex powerspace $\CC(\Sigma L)$ is not quasicontinuous; this failure then transfers to the convex powerdomain $\PC(L)$. This answers the question of preservation of quasicontinuity by convex powerdomains negatively.

Moreover, within the framework of directed spaces, we also prove that directed convex powerspaces reflect both continuity and quasicontinuity, and that directed upper powerspaces reflect quasicontinuity.  These results imply that convex powerdomains reflect both continuity and quasicontinuity, and that upper powerdomains reflect quasicontinuity.  Together with the known results for lower and upper directed spaces and powerdomains, these arguments give the complete preservation and reflection profile of the lower, upper and convex powerdomains for continuity and quasicontinuity. The three kinds of powerdomains have the same preservation profile for continuity, but they have different preservation and reflection profiles for quasicontinuity. The final comparison is displayed in the following table.

\begin{center}
\begingroup
\renewcommand{\arraystretch}{1.15}
\begin{tabular}{c|cc|cc|cc}
 & \multicolumn{2}{c|}{$\LL/\PL$} & \multicolumn{2}{c|}{$\UU/\PU$} & \multicolumn{2}{c}{$\CC/\PC$}\\
\text{property} & \text{pres.} & \text{refl.} & \text{pres.} & \text{refl.} & \text{pres.} & \text{refl.}\\
\hline
\text{continuity} & \text{yes} & \text{yes} & \text{yes} & \underline{no} & \text{yes} & \textbf{yes}\\
\text{quasicontinuity} & \text{yes} & \text{yes} & \underline{no} & \textbf{yes} & \textbf{no} & \textbf{yes}
\end{tabular}
\endgroup

\smallskip
\small Preservation and reflection of continuity and quasicontinuity for lower, upper and convex powerspaces/powerdomains. The underlined entries are the results in \cite{ChenKouLyu2024}, and the bold entries are the results in this paper.
\end{center}

In particular, we show $\PC(L)$ is continuous if and only if $L$ is continuous.  We also explain why this continuity theorem should not be confused with compact-lens representability.  Unlike the upper-powerdomain situation, a bare order isomorphism between $\PC(L)$ and a compact-lens poset does not imply continuity.  To see this, we exhibit the two-chain dcpo with a common top,
\[
        Y = \{a_n : n \geq 1\} \cup \{b_n : n \geq 1\} \cup\{\omega\},
\]
which is quasicontinuous but not continuous, and prove that $\PC(Y)$ is quasicontinuous but not
continuous. This example shows that, even when a compact-lens poset is abstractly order isomorphic to $\PC(L)$, this should not be confused with a canonical representation strong enough to detect continuity.

The paper is organized as follows.  Section~2 collects the basic notions, including the published directed upper and lower powerspaces.  Section~3 introduces and constructs the directed convex powerspace.  Section~4 proves the retract theorems and uses the upper retract to obtain the main counterexample.  Section~5 proves and summarizes the remaining preservation and reflection results for continuity and quasicontinuity.  Section~6 returns to the retract at the dcpo level and explains why the directed-space formulation is technically simpler.

\section{Preliminaries}

All topological spaces under consideration are assumed to satisfy the $T_0$ separation axiom.  Order-theoretic terminology for a space always refers to its specialization order.  The directed-space viewpoint used below follows the framework of directed and monotone-convergence spaces developed in \cite{Erne2009,KeimelLawson2009,LuoXu2017,Xie2021}; for broader background on topological domain theory, see also \cite{GoubaultLarrecq2013}.

We use directed spaces as a topological extension framework for dcpos.  Every dcpo $L$ gives a directed space $\Sigma L$ by taking its Scott topology, while the Scott completion of a directed space produces a dcpo with its Scott space.  Thus one may first build finite-generator power constructions in $\Dtop$ and then recover the usual dcpo powerdomains by Scott completion.  This separation is useful because the directed powerspace often records the essential convergence data before directed suprema are freely added.

The lower, upper and convex carrier sets in $\Dtop$ are denoted by $L(X)$, $U(X)$ and $C(X)$, respectively.  After equipping these carrier sets with their directed powerspace topologies and semilattice operations, the resulting directed powerspaces are denoted by $\LL(X)$, $\UU(X)$ and $\CC(X)$.  When $X=\Sigma L$ is the Scott space of a dcpo, their Scott completions are the corresponding dcpo powerdomains $\PL(L)$, $\PU(L)$ and $\PC(L)$.  Moreover, continuity and quasicontinuity are preserved and reflected by these completions.  Hence the problems of whether $\PL(L)$, $\PU(L)$ or $\PC(L)$ is continuous or quasicontinuous may be transferred to the corresponding directed powerspace over $\Sigma L$.

\begin{defin}[Directed space]
A space $X$ is termed a \emph{directed space} if its topology is determined by convergence of directed subsets: a subset $U\subseteq X$ is open precisely when, for every directed $D\subseteq X$ converging to $x\in U$, one has $D\cap U\neq\varnothing$.  We write $\Dtop$ for the category of directed spaces and continuous maps.  The categorical product in $\Dtop$ is denoted by $\otimes$.
\end{defin}

\begin{defin}[Continuity and quasicontinuity for directed spaces]\label{def:directed-continuity}
Let $X$ be a directed space.  For $x,y\in X$, write $x\ll_d y$ if, for every directed subset $D\subseteq X$ with $D\to y$, there exists $d\in D$ such that $x\leq d$ in the specialization order of $X$.

The directed space $X$ is \emph{continuous} if every point is the limit of a directed family of elements way-below it.  Equivalently, $X$ is a $c$-space: for every open set $U$ and every $x\in U$, there exists $u\in U$ such that
\[
        x\in\intr(\upar u)\subseteq U.
\]
The directed space $X$ is \emph{quasicontinuous} if, for every open set $U$ and every $x\in U$, there exists a finite set $F\subseteq U$ such that
\[
        x\in\intr(\upar F)\subseteq U.
\]
For a dcpo $L$, these notions applied to the Scott space $(L,\sig(L))$ coincide with the usual notions of continuity and quasicontinuity for dcpos.
\end{defin}

If $D$ is a directed subset of a directed space $X$, we write
\[
        \Lim(D)=\{x\in X:D\to x\}
\]
for the set of its topological limits.  This notation is used in the convergence descriptions of the directed powerspaces below.

\begin{defin}[Retract]
Let $\mathcal C$ be a category.  An object $A$ is a \emph{retract} of an object $B$ if there are morphisms $e\colon A\to B$ and $r\colon B\to A$ such that $r\circ e=\id_A$.  In $\Dtop$ this means that $e$ and $r$ are continuous maps of directed spaces.
\end{defin}

\begin{defin}[Directed semilattice]
A \emph{directed semilattice} is a directed space $Y$ equipped with a continuous binary map
\[
        +\colon Y\otimes Y\to Y
\]
satisfying
\[
        y+y=y,\qquad (x+y)+z=x+(y+z),\qquad x+y=y+x .
\]
A homomorphism of directed semilattices is a continuous map that preserves the binary operation $+$.
\end{defin}

\begin{prop}[Continuity via a determining convergence class]\label{prop:conv-continuity-criterion}
Let $X$ be a set equipped with a convergence class $\mathcal S$, where $(A,x)\in\mathcal S$ is written $A\to_{\mathcal S}x$.  Let $\mathcal{O}_{\mathcal S}(X)$ be the topology determined by $\mathcal S$, that is, $U\subseteq X$ is open if and only if
\[
        A\to_{\mathcal S}x\in U
        \quad\Longrightarrow\quad
        A\cap U\neq\varnothing
\]
for every $(A,x)\in\mathcal S$.  If $Y$ is a topological space and $h\colon (X,\mathcal{O}_{\mathcal S}(X))\to Y$ is a map such that
\[
        A\to_{\mathcal S}x
        \quad\Longrightarrow\quad
        h(A)\to h(x)
\]
in $Y$, then $h$ is continuous.
\end{prop}

\begin{proof}
Let $V\subseteq Y$ be open.  We prove that $h^{-1}(V)\in\mathcal{O}_{\mathcal S}(X)$.  Suppose $A\to_{\mathcal S}x$ and $x\in h^{-1}(V)$.  Then $h(x)\in V$.  By hypothesis, $h(A)\to h(x)$ in $Y$.  Since $V$ is an open neighbourhood of $h(x)$, the convergence $h(A)\to h(x)$ implies
\[
        h(A)\cap V\neq\varnothing.
\]
Consequently, there exists $a\in A$ with $h(a)\in V$, equivalently $a\in A\cap h^{-1}(V)$.  Thus
\[
        A\cap h^{-1}(V)\neq\varnothing.
\]
By the definition of $\mathcal{O}_{\mathcal S}(X)$, the set $h^{-1}(V)$ is open.  Since this holds for arbitrary open $V\subseteq Y$, $h$ is continuous.
\end{proof}

\begin{defin}[Powerdomains in $\DCPO$]
A \emph{dcpo-semilattice} is a dcpo $A$ equipped with a Scott-continuous binary operation $+$ satisfying idempotency, associativity and commutativity.  It is \emph{inflationary} if $x\leq x+y$ for all $x,y$, and \emph{deflationary} if $x+y\leq x$ for all $x,y$.  The \emph{lower powerdomain} $\PL(L)$ is the free inflationary dcpo-semilattice over $L$, the \emph{upper powerdomain} $\PU(L)$ is the free deflationary dcpo-semilattice over $L$, and the \emph{convex powerdomain} $\PC(L)$ is the free dcpo-semilattice over $L$ \cite{Plotkin1976,Smyth1978,Heckmann1991,AbramskyJung1994}.
\end{defin}

\begin{defin}[Scott completion, or $D$-completion]
The \emph{Scott completion}, equivalently the $D$-completion in the present directed-space setting, of a directed space $X$ is a dcpo $X^d$ together with a continuous map $j\colon X\to\Sigma X^d$, where $\Sigma X^d$ denotes the Scott space of $X^d$, satisfying the usual universal property: every continuous map from $X$ into the Scott space of a dcpo extends uniquely to a Scott-continuous map out of $X^d$.  In this paper we use Scott completions through this universal property and through the representation theorem below.
\end{defin}

\begin{theor}[$D$-completion representation {\cite{KeimelLawson2009,KeimelLawson2009Ops,ChenKouLyu2024,ChenKouLyuXie2024Cones}}]\label{the:scott-completion}
For every dcpo $L$, the lower and upper powerdomains $\PL(L)$ and $\PU(L)$ are the Scott completions of the directed lower and upper powerspaces $\LL(\Sigma L)$ and $\UU(\Sigma L)$, respectively.  The same completion principle applies to free dcpo-algebras: they are obtained by Scott completing the corresponding free directed-space algebras over $\Sigma L$.  In particular, after the construction of $\CC(\Sigma L)$ below, its Scott completion is the convex powerdomain $\PC(L)$.
\end{theor}

\begin{note}[Continuity, quasicontinuity and Scott completion {\cite{KeimelLawson2009,Xie2021,ChenKouLyu2024}}]\label{the:qc-completion}
For the directed powerstructures considered here, continuity and quasicontinuity are preserved and reflected by Scott completion.  Thus the corresponding directed powerspace and dcpo powerdomain have the same status with respect to these two properties.  In particular, for every dcpo $L$,
\[
        \PC(L)\text{ is quasicontinuous}
        \quad\Longleftrightarrow\quad
        \CC(\Sigma L)\text{ is quasicontinuous},
\]
and similarly
\[
        \PU(L)\text{ is quasicontinuous}
        \quad\Longleftrightarrow\quad
        \UU(\Sigma L)\text{ is quasicontinuous}.
\]
The same equivalences hold with ``quasicontinuous'' replaced by ``continuous'', and also for the lower powerspace $\LL(\Sigma L)$ and the lower powerdomain $\PL(L)$.
\end{note}

\begin{theor}[Upper powerdomains of quasicontinuous dcpos {\cite{ChenKouLyu2024}}]\label{the:upper-equivalences}
Let $L$ be a quasicontinuous dcpo.  Then the following are equivalent:
\begin{enumerate}[label=(\roman*)]
    \item $\PU(L)$ is quasicontinuous;
    \item $\PU(L)$ is continuous;
    \item the topology of $\UU(\Sigma L)$ coincides with the Scott topology of its specialization order;
    \item $\PU(L)\cong\QQ(L)$, the semilattice of nonempty Scott compact saturated subsets of $L$.
\end{enumerate}
\end{theor}

The directed lower and upper powerspaces recalled below are published directed-space analogues of the Hoare and Smyth powerdomains; see \cite{Xie2021,ChenKouXie2024RMJ,ChenKouLyu2024}.  We use only their finite-generator descriptions and their basic representation properties.

\begin{defin}[Directed upper powerspace {\cite{Xie2021,ChenKouLyu2024}}]\label{def:upper-convergence}
Let $X$ be a directed space and set
\[
        U(X)=\{\upar F:F\subseteq_{\fin}X,\ F\neq\varnothing\},
\]
ordered by reverse inclusion:
\[
        \upar F_1\leq_U\upar F_2
        \quad\Longleftrightarrow\quad
        \upar F_2\subseteq\upar F_1.
\]
For a directed subset $\mathcal E\subseteq U(X)$ and a finite nonempty set $F\subseteq X$, we write
\[
        \mathcal E\Rightarrow_U\upar F
\]
provided there exist directed subsets $D_1,\ldots,D_k$ of $X$, with $k\geq 1$, satisfying:
\begin{enumerate}[label=(\roman*)]
    \item $F\cap\Lim(D_i)\neq\varnothing$ for $1\leq i\leq k$;
    \item $F\subseteq\bigcup_{i=1}^k\Lim(D_i)$;
    \item for every $(d_1,\ldots,d_k)\in D_1\times\cdots\times D_k$, there exists $\upar G\in\mathcal E$ with
\[
        \upar G\subseteq\bigcup_{i=1}^k\upar d_i .
\]
\end{enumerate}
A subset $V\subseteq U(X)$ is $\Rightarrow_U$-\emph{open} if
\[
        \mathcal E\Rightarrow_U\upar F\in V
        \quad\Longrightarrow\quad
        \mathcal E\cap V\neq\varnothing
\]
for every directed $\mathcal E\subseteq U(X)$.  Let $\mathcal{O}_{\Rightarrow_U}(U(X))$ denote the family of all $\Rightarrow_U$-open subsets.  The \emph{directed upper powerspace} $\UU(X)$ is $U(X)$ equipped with this topology and the union operation.  This construction yields the free directed deflationary semilattice over $X$.
\end{defin}

\begin{defin}[Directed lower powerspace {\cite{Xie2021}}]\label{def:lower-convergence}
Let $X$ be a directed space and set
\[
        L(X)=\{\down F:F\subseteq_{\fin}X,\ F\neq\varnothing\},
\]
ordered by inclusion:
\[
        \down F_1\leq_L\down F_2
        \quad\Longleftrightarrow\quad
        \down F_1\subseteq\down F_2 .
\]
For a directed subset $\mathcal A\subseteq L(X)$ and a finite nonempty set $F\subseteq X$, we write
\[
        \mathcal A\Rightarrow_L\down F
\]
provided that, for every $a\in F$, there exists a directed subset $D_a$ of $X$ such that
\[
        a\in\Lim(D_a)
        \quad\text{and}\quad
        D_a\subseteq\bigcup\{\down G:\down G\in\mathcal A\}.
\]
Let $\mathcal{O}_{\Rightarrow_L}(L(X))$ denote the topology determined by this convergence.  We write $\LL(X)$ for the resulting directed lower powerspace, equipped with the union operation.  It is the free directed inflationary semilattice over $X$.
\end{defin}

\section{Directed Convex Powerspaces}

We first give a simple topological representation of directed convex powerspaces, which is the directed-space analogue of the Plotkin convex powerdomain.  The construction is based on finite generators, and the topology is determined by a convergence class.  The resulting space is a directed semilattice, and it satisfies the universal property of the free directed semilattice over a directed space.  The Scott completion of this directed convex powerspace over the Scott space of a dcpo $L$ is the corresponding convex powerdomain.

\begin{defin}[Directed convex powerspace]
Let $X$ be a directed space.  A \emph{directed convex powerspace} of $X$ is a directed semilattice $\CC(X)$, together with a continuous map $i_X\colon X\to \CC(X)$, such that for every directed semilattice $Y$ and every continuous map $f\colon X\to Y$, there exists a unique continuous semilattice homomorphism $\bar f\colon \CC(X)\to Y$ satisfying $\bar f\circ i_X=f$.
\end{defin}

\begin{defin}[Finite-lens construction]\label{def:convex-convergence}
Let $X$ be a directed space.  For a nonempty finite subset $F\subseteq X$, we define
\[
        \widehat F=(\down F,\upar F),
        \qquad
        C(X)=\{\widehat F:F\subseteq_{\fin} X,\ F\neq\varnothing\}.
\]
We define an order on $C(X)$ by
\[
        \widehat F_1\leq_C\widehat F_2
        \quad\Longleftrightarrow\quad
        \down F_1\subseteq\down F_2
        \ \text{and}\
        \upar F_2\subseteq\upar F_1.
\]
If $\mathcal D\subseteq C(X)$ is directed and $F\subseteq_{\fin}X$ is nonempty, we write
\[
        \mathcal D\Rightarrow_C\widehat F
\]
provided there exist directed subsets $D_1,\ldots,D_k$ of $X$, with $k\geq 1$, satisfying:
\begin{enumerate}[label=(\roman*)]
    \item $D_i\subseteq\bigcup\{\down G:\widehat G\in\mathcal D\}$ for $1\leq i\leq k$;
    \item $F\cap\Lim(D_i)\neq\varnothing$ for $1\leq i\leq k$;
    \item $F\subseteq\bigcup_{i=1}^k\Lim(D_i)$;
    \item for every $(d_1,\ldots,d_k)\in D_1\times\cdots\times D_k$, there exists $\widehat G\in\mathcal D$ such that
    \[
            \upar G\subseteq\bigcup_{i=1}^k\upar d_i .
    \]
\end{enumerate}
A subset $U\subseteq C(X)$ is called $\Rightarrow_C$-\emph{open} if
\[
        \mathcal D\Rightarrow_C\widehat F\in U
        \quad\Longrightarrow\quad
        \mathcal D\cap U\neq\varnothing
\]
for every directed $\mathcal D\subseteq C(X)$.  Let $\mathcal{O}_{\Rightarrow_C}(C(X))$ denote the family of all $\Rightarrow_C$-open subsets.
\end{defin}

\begin{prop}\label{prop:convex-topology}
$(C(X),\mathcal{O}_{\Rightarrow_C}(C(X)))$ is a directed space, and its specialization order coincides with $\leq_C$.
\end{prop}

\begin{proof}
The empty set and $C(X)$ are $\Rightarrow_C$-open.  If $U,V$ are $\Rightarrow_C$-open and $\mathcal D\Rightarrow_C\widehat F\in U\cap V$, choose $\widehat G_U\in\mathcal D\cap U$ and $\widehat G_V\in\mathcal D\cap V$.  Since $\mathcal D$ is directed, there exists $\widehat G\in\mathcal D$ with $\widehat G_U,\widehat G_V\leq_C\widehat G$.  Every $\Rightarrow_C$-open set is an upper set for $\leq_C$, because $\widehat F_1\leq_C\widehat F_2$ implies $\{\widehat F_2\}\Rightarrow_C\widehat F_1$: if $F_1=\{a_1,\ldots,a_n\}$, the singleton directed sets $\{a_1\},\ldots,\{a_n\}$ satisfy the four clauses in Definition~\ref{def:convex-convergence}.  Hence $\widehat G\in U\cap V$.  Arbitrary unions are immediate from the definition, so $\mathcal{O}_{\Rightarrow_C}(C(X))$ is a topology.

The same observation shows that if $\widehat F_1\leq_C\widehat F_2$, then every open neighbourhood of $\widehat F_1$ contains $\widehat F_2$, so $\widehat F_1$ is below $\widehat F_2$ in the specialization order.  Conversely, for a fixed $\widehat F_2$, the set
\[
        \{\widehat F:\widehat F\leq_C\widehat F_2\}
\]
is closed.  Indeed, suppose $\mathcal D\Rightarrow_C\widehat G$, witnessed by $D_1,\ldots,D_k$, and every member of $\mathcal D$ is $\leq_C\widehat F_2$.  Then each $D_i$ is contained in $\down F_2$.  Since $\down F_2$ is closed and every $g\in G$ is a limit of some $D_i$, we have $G\subseteq\down F_2$, and hence $\down G\subseteq\down F_2$.

It remains to show $\upar F_2\subseteq\upar G$.  Let $y\in\upar F_2$.  If $y\notin\upar G$, then for every $i$ we may choose $g_i\in G\cap\Lim(D_i)$ and have $g_i\not\leq y$.  The set $X\setminus\down y$ is open, so we can choose $d_i\in D_i\cap(X\setminus\down y)$ for each $i$.  By the fourth clause of $\Rightarrow_C$, there exists $\widehat H\in\mathcal D$ such that
\[
        \upar H\subseteq\bigcup_{i=1}^k\upar d_i .
\]
Since $\widehat H\leq_C\widehat F_2$, we have $y\in\upar F_2\subseteq\upar H$, and therefore $y\in\upar d_i$ for some $i$.  This contradicts $d_i\not\leq y$.  Thus $y\in\upar G$, proving $\widehat G\leq_C\widehat F_2$.  Therefore the specialization order is precisely $\leq_C$.

Finally, by definition $\mathcal D\Rightarrow_C\widehat F$ implies convergence of $\mathcal D$ to $\widehat F$ in $\mathcal{O}_{\Rightarrow_C}(C(X))$.  Hence every directed-open subset is $\Rightarrow_C$-open.  The reverse inclusion holds automatically for every topology generated in this manner.  Consequently, the topology is directed, and $(C(X),\mathcal{O}_{\Rightarrow_C}(C(X)))$ is a directed space.
\end{proof}

For the operation in this semilattice, we define it as:
\[
        \widehat F\oplus\widehat G=\widehat{F\cup G}.
\]
We now show that this operation is continuous, and $(C(X),\mathcal{O}_{\Rightarrow_C}(C(X)),\oplus)$ is a directed semilattice.

\begin{prop}\label{prop:semilattice}
$(C(X),\mathcal{O}_{\Rightarrow_C}(C(X)),\oplus)$ is a directed semilattice.
\end{prop}

\begin{proof}
Idempotency, associativity and commutativity follow from the corresponding laws for union.  It remains to establish continuity.  By the product property in $\Dtop$, it suffices to prove continuity in each variable.  Fix $\widehat G$.  The map
\[
        T_G(\widehat F)=\widehat{F\cup G}
\]
is monotone.  Suppose $\mathcal D\Rightarrow_C\widehat F$, witnessed by directed sets $D_1,\ldots,D_k$, and write $G=\{b_1,\ldots,b_m\}$.  Adjoin the singleton directed sets $\{b_1\},\ldots,\{b_m\}$.  The original witnesses still satisfy the lower-coordinate condition after applying $T_G$, and each new singleton is contained in every lower component $\down(H\cup G)$ with $\widehat H\in\mathcal D$.  The limit-covering clauses are immediate for $F\cup G$.  Finally, for a tuple $(d_1,\ldots,d_k,b_1,\ldots,b_m)$, choose $\widehat H\in\mathcal D$ with
\[
        \upar H\subseteq\bigcup_{i=1}^k\upar d_i .
\]
Then
\[
        \upar(H\cup G)
        \subseteq
        \bigcup_{i=1}^k\upar d_i\cup\bigcup_{j=1}^m\upar b_j .
\]
Thus $T_G(\mathcal D)\Rightarrow_C\widehat{F\cup G}$.  By Proposition~\ref{prop:conv-continuity-criterion}, $T_G$ is continuous.  Symmetry yields continuity in the other variable.
\end{proof}

The following theorem establishes that the constructed directed semilattices are exactly the directed convex powerspaces of directed spaces.
Hence, the directed convex powerspaces have a simple topological representation, which can be used to study the properties of directed convex powerspaces and convex powerdomains.

\begin{theor}\label{the:directed-convex-free}
For every directed space $X$, the directed semilattice $(C(X),\mathcal{O}_{\Rightarrow_C}(C(X)),\oplus)$, with unit
\[
        i_X(x)=\widehat{\{x\}},
\]
is the directed convex powerspace of $X$.
\end{theor}

\begin{proof}
The map $i_X$ is continuous.  Indeed, if $D\to x$ is directed in $X$, then $\{\widehat{\{d\}}:d\in D\}\Rightarrow_C\widehat{\{x\}}$, and Proposition~\ref{prop:conv-continuity-criterion} yields continuity.

Let $Y$ be a directed semilattice and let $f\colon X\to Y$ be continuous.  Define
\[
        \bar f(\widehat F)=\sum_{a\in F} f(a),
\]
where the sum is taken with respect to the semilattice operation of $Y$.  We first note that this is well-defined and monotone.  Suppose $\widehat F_1\leq_C\widehat F_2$, say $F_1=\{a_1,\ldots,a_m\}$ and $F_2=\{b_1,\ldots,b_n\}$.  From $\down F_1\subseteq\down F_2$, choose, for each $a_i$, an element $b_i'\in F_2$ with $a_i\leq b_i'$.  Put $R=F_2\setminus\{b_1',\ldots,b_m'\}$.  For every $b\in R$, the inclusion $\upar F_2\subseteq\upar F_1$ gives $a_b\in F_1$ with $a_b\leq b$.  Using monotonicity, idempotency and commutativity of the semilattice operation,
\[
        \sum_{a\in F_1}f(a)
        =
        \sum_{a\in F_1}f(a)+\sum_{b\in R}f(a_b)
        \leq
        \sum_{i=1}^m f(b_i')+\sum_{b\in R} f(b)
        =
        \sum_{b\in F_2}f(b).
\]
Thus $\widehat F_1\leq_C\widehat F_2$ implies $\bar f(\widehat F_1)\leq\bar f(\widehat F_2)$.  If $\widehat F_1=\widehat F_2$, the reverse inequality follows as well, so $\bar f$ is independent of the chosen finite generator.  Clearly $\bar f\circ i_X=f$ and $\bar f(\widehat F\oplus\widehat G)=\bar f(\widehat F)+\bar f(\widehat G)$.

It remains to prove continuity of $\bar f$.  Suppose $\mathcal D\Rightarrow_C\widehat F$, witnessed by $D_1,\ldots,D_k$.  For each $i$, choose $b_i\in F\cap\Lim(D_i)$, and for each $a\in F\setminus\{b_1,\ldots,b_k\}$ choose an index $i(a)$ with $a\in\Lim(D_{i(a)})$.  Repetitions among the chosen $b_i$ are harmless by idempotency, and the second sum below is omitted if the displayed set is empty.  Since $f$ is continuous and the semilattice operation of $Y$ is continuous, the directed family
\[
        \left\{
        \sum_{i=1}^k f(d_i)+
        \sum_{a\in F\setminus\{b_1,\ldots,b_k\}} f(e_a):
        d_i\in D_i,\ e_a\in D_{i(a)}
        \right\}
        \longrightarrow
        \sum_{i=1}^k f(b_i)+
        \sum_{a\in F\setminus\{b_1,\ldots,b_k\}} f(a)
        =
        \bar f(\widehat F).
\]
Let $U$ be an open neighbourhood of $\bar f(\widehat F)$.  Choose $d_i\in D_i$ and $e_a\in D_{i(a)}$ such that the displayed sum belongs to $U$.  By directedness of the $D_i$, enlarge the $d_i$'s so that $e_a\leq d_{i(a)}$ for all $a$.  Since $U$ is upper, we then have
\[
        \sum_{i=1}^k f(d_i)\in U .
\]
By the fourth clause in Definition~\ref{def:convex-convergence}, choose $\widehat G_0\in\mathcal D$ with
\[
        \upar G_0\subseteq\bigcup_{i=1}^k\upar d_i .
\]
By the first clause, for each $i$ choose $\widehat G_i\in\mathcal D$ such that $d_i\in\down G_i$.  Since $\mathcal D$ is directed, choose $\widehat G\in\mathcal D$ above $\widehat G_0,\widehat G_1,\ldots,\widehat G_k$.  Then
\[
        \widehat{\{d_1,\ldots,d_k\}}\leq_C\widehat G .
\]
Indeed, the lower inclusion follows from $d_i\in\down G_i\subseteq\down G$, while the upper inclusion follows from $\upar G\subseteq\upar G_0\subseteq\bigcup_i\upar d_i$.
By monotonicity of $\bar f$ and upperness of $U$, $\bar f(\widehat G)\in U$.  Hence $\bar f(\mathcal D)\to\bar f(\widehat F)$, and Proposition~\ref{prop:conv-continuity-criterion} gives continuity of $\bar f$.

Uniqueness is forced by the unit and the semilattice law: every $\widehat F$ is the finite sum of the generators $\widehat{\{a\}}$ with $a\in F$.  Thus any continuous semilattice homomorphism extending $f$ must equal $\bar f$.
\end{proof}

We henceforth write
\[
        \CC(X)=(C(X),\mathcal{O}_{\Rightarrow_C}(C(X)),\oplus)
\]
and call it the \emph{directed convex powerspace} of $X$.  By Theorem~\ref{the:scott-completion}, the Scott completion of $\CC(\Sigma L)$ is the convex powerdomain $\PC(L)$ of a dcpo $L$.

\section{Retractions and the Counterexample}

The first use of directed convex powerspaces is the construction of canonical retracts onto the directed upper and lower powerspaces recalled in Section~2.  The upper retract is the mechanism that transfers the known upper-powerdomain counterexample to the convex setting.

\begin{theor}\label{the:directed-retract}
Let $X$ be a directed space with a greatest element $\Top$.  Then $\UU(X)$ is a retract of $\CC(X)$.
\end{theor}

\begin{proof}
Define
\[
        e\colon \UU(X)\to\CC(X),
        \qquad
        e(\upar F)=\widehat{F\cup\{\Top\}},
\]
and
\[
        r\colon \CC(X)\to\UU(X),
        \qquad
        r(\widehat F)=\upar F.
\]
We first show that $e$ and $r$ are well-defined. For $r$, we need only show that if two nonempty finite subsets $F_1$ and $F_2$ are such that $\widehat F_1=\widehat F_2$, then $r(\widehat F_1)=r(\widehat F_2)$, that is, $\upar F_1=\upar F_2$. This is immediate. For $e$, we need to show that if two nonempty finite subsets $F_1$ and $F_2$ are such that $\upar F_1=\upar F_2$, then $e(\upar F_1)=e(\upar F_2)$, that is, $\widehat{F_1\cup\{\Top\}}=\widehat{F_2\cup\{\Top\}}$. This is also true, since $\down(F_1\cup\{\Top\})=\down(F_2\cup\{\Top\})=X$ and $\upar(F_1\cup\{\Top\}) = \upar F_1 = \upar F_2 = \upar(F_2\cup\{\Top\})$.

We next prove that $r$ is continuous.  It is monotone: if $\widehat F_1\leq_C\widehat F_2$, then $\upar F_2\subseteq\upar F_1$, hence $r(\widehat F_1)\leq_U r(\widehat F_2)$.  Suppose $\mathcal D\Rightarrow_C\widehat F$ in $\CC(X)$, witnessed by $D_1,\ldots,D_k$.  The last three clauses of Definition~\ref{def:convex-convergence} are exactly the upper convergence clauses of Definition~\ref{def:upper-convergence} for $r(\mathcal D)\Rightarrow_U\upar F$.  Thus $r$ preserves the specified convergence, and Proposition~\ref{prop:conv-continuity-criterion} implies that $r$ is continuous.

Next, we prove that $e$ is continuous.  It is monotone: if $\upar F_1\leq_U\upar F_2$, then $\upar F_2\subseteq\upar F_1$, while
\[
        \down(F_1\cup\{\Top\})=\down(F_2\cup\{\Top\})=X,
\]
so $e(\upar F_1)\leq_C e(\upar F_2)$.  Suppose $\mathcal E\Rightarrow_U\upar F$, witnessed by $D_1,\ldots,D_k$.  Adjoin the singleton directed set $D_{k+1}=\{\Top\}$.  Since every member of $e(\mathcal E)$ has first coordinate $X$, all these directed sets lie in the union of the first coordinates of $e(\mathcal E)$.  The limit-covering clauses hold for $F\cup\{\Top\}$ by the upper convergence of $\mathcal E$ and by $\Top\in\Lim(\{\Top\})$.  For every $(d_1,\ldots,d_k,\Top)$, the upper-convergence condition provides $\upar G\in\mathcal E$ with
\[
        \upar G\subseteq\bigcup_{i=1}^k\upar d_i
        \subseteq\bigcup_{i=1}^k\upar d_i\cup\upar\Top .
\]
But
\[
        \upar(G\cup\{\Top\})=\upar G,
\]
because $\Top\in\upar G$.  Hence $e(\mathcal E)\Rightarrow_C\widehat{F\cup\{\Top\}}$.  By Proposition~\ref{prop:conv-continuity-criterion}, $e$ is continuous.

Finally,
\[
        (r\circ e)(\upar F)
        =r(\widehat{F\cup\{\Top\}})
        =\upar(F\cup\{\Top\})
        =\upar F,
\]
again because $\Top\in\upar F$.  Therefore $r\circ e=\id_{\UU(X)}$, and $\UU(X)$ is a retract of $\CC(X)$ in $\Dtop$.
\end{proof}

The following theorem is the lower-power analogue of Theorem~\ref{the:directed-retract}.  Its proof is similar, and we only give a sketch.

\begin{theor}[Lower retraction]\label{the:directed-lower-retract}
Let $X$ be a directed space possessing a least element $\Bot$.  Then $\LL(X)$ is a retract of $\CC(X)$ in $\Dtop$.
\end{theor}

\begin{proof}
Define
\[
        e_L\colon \LL(X)\to\CC(X),
        \qquad
        e_L(\down F)=\widehat{F\cup\{\Bot\}},
\]
and
\[
        r_L\colon \CC(X)\to\LL(X),
        \qquad
        r_L(\widehat F)=\down F .
\]
Analogously to the directed upper powerspace case, both maps are well-defined and continuous, and $r_L \circ e_L = \id_{\LL(X)}$. Therefore, $\LL(X)$ is a retract of $\CC(X)$ in $\Dtop$.

\end{proof}

\begin{ex}\label{ex:no-top}
        In Theorem~\ref{the:directed-retract}, the hypothesis that $X$ has a greatest element is essential.  Consider
\[
        X=\{a,b\}
\]
equipped with the two-point discrete topology and the equality specialization order.  Then
\[
        C(X)=\{\widehat{\{a\}},\widehat{\{b\}},\widehat{\{a,b\}}\}.
\]
Since $\down F=\upar F=F$ for every $F\subseteq X$, the order $\leq_C$ on $C(X)$ is equality.  Hence $\CC(X)$ is a three-point discrete directed space.

On the other hand,
\[
        U(X)=\{\upar\{a\},\upar\{b\},\upar\{a,b\}\}
        =\{\{a\},\{b\},\{a,b\}\},
\]
ordered by reverse inclusion.  Thus $\{a,b\}\leq_U\{a\}$ and $\{a,b\}\leq_U\{b\}$, while $\{a\}$ and $\{b\}$ are incomparable.  Its topology is the upper-set topology for this finite directed space, so the singleton $\{\{a,b\}\}$ is not open.  In particular, $\UU(X)$ is not discrete.

\begin{center}
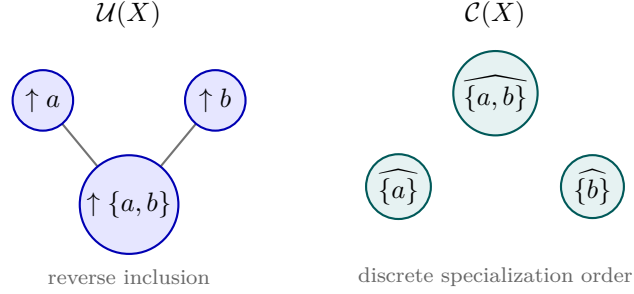

\begin{minipage}{0.82\textwidth}
\centering
\begin{tikzpicture}[
    scale=0.95,
    every node/.style={font=\small},
    point/.style={circle, draw=#1!70!black, fill=#1!10, thick, minimum size=8mm, inner sep=1.5pt},
    title/.style={font=\small\bfseries},
    edge/.style={line width=0.8pt, draw=black!55}
]
    \coordinate (uabp) at (0,0);
    \coordinate (uap) at (-1.2,1.45);
    \coordinate (ubp) at (1.2,1.45);
    \draw[edge] (uabp) -- (uap);
    \draw[edge] (uabp) -- (ubp);
    \node[title] at (0,2.65) {$\UU(X)$};
    \node[point=blue] (uab) at (uabp) {$\upar\{a,b\}$};
    \node[point=blue] (ua) at (uap) {$\upar a$};
    \node[point=blue] (ub) at (ubp) {$\upar b$};
    \node[font=\scriptsize, text=black!60] at (0,-1) {reverse inclusion};

    \node[title] at (5.1,2.65) {$\CC(X)$};
    \node[point=teal] at (3.75,0.25) {$\widehat{\{a\}}$};
    \node[point=teal] at (6.45,0.25) {$\widehat{\{b\}}$};
    \node[point=teal] at (5.1,1.55) {$\widehat{\{a,b\}}$};
    \node[font=\scriptsize, text=black!60] at (5.1,-1) {discrete specialization order};
\end{tikzpicture}

\captionof{figure}{Without a greatest element, $\UU(X)$ is a non-discrete three-point $V$, while $\CC(X)$ is discrete.}
\end{minipage}
\end{center}

Suppose that $\UU(X)$ were a retract of $\CC(X)$.  Then there would exist continuous maps
\[
        e\colon \UU(X)\to\CC(X),\qquad r\colon \CC(X)\to\UU(X)
\]
with $r\circ e=\id_{\UU(X)}$.  The map $e$ is injective, and both spaces have three points; hence $e$ is bijective and $r=e^{-1}$.  Since both $e$ and $r$ are continuous, this would make $\UU(X)$ homeomorphic to $\CC(X)$, contradicting the fact that $\CC(X)$ is discrete and $\UU(X)$ is not.  Thus the retract theorem fails for this space without a top element.

If one adjoins a greatest element, the obstruction disappears.  Let
\[
        X^\Top=\{a,b,\Top\}
\]
with $a,b\leq\Top$ and $a,b$ incomparable.  Then Theorem~\ref{the:directed-retract} applies.  Explicitly,
\[
        e(\upar F)=\widehat{F\cup\{\Top\}},
        \qquad
        r(\widehat F)=\upar F.
\]
The picture below shows how the four-point upper powerspace becomes the highlighted four-point retract inside the seven-point convex powerspace.

\begin{center}
\begin{minipage}{0.9\textwidth}
\centering
\begin{tikzpicture}[
    scale=0.92,
    every node/.style={font=\small},
    upnode/.style={circle, draw=blue!70!black, fill=blue!10, thick, minimum size=8mm, inner sep=1.5pt},
    cnode/.style={circle, draw=black!45, fill=black!4, thick, minimum size=8mm, inner sep=1.5pt},
    imnode/.style={circle, draw=teal!65!black, fill=teal!12, very thick, minimum size=8mm, inner sep=1.5pt},
    title/.style={font=\small\bfseries},
    edge/.style={line width=0.8pt, draw=black!55},
    imedge/.style={line width=1.2pt, draw=teal!65!black}
]
    \coordinate (uxp) at (0,0);
    \coordinate (ua2p) at (-1.25,1.45);
    \coordinate (ub2p) at (1.25,1.45);
    \coordinate (utp) at (0,2.9);
    \coordinate (cabp) at (6.1,-0.3);
    \coordinate (cap) at (4.55,0.75);
    \coordinate (cbp) at (7.65,0.75);
    \coordinate (cabtp) at (6.1,1.45);
    \coordinate (catp) at (4.75,2.45);
    \coordinate (cbtp) at (7.45,2.45);
    \coordinate (ctp) at (6.1,3.35);

    \draw[edge] (uxp) -- (ua2p) -- (utp);
    \draw[edge] (uxp) -- (ub2p) -- (utp);
    \draw[edge] (cabp) -- (cabtp);
    \draw[edge] (cap) -- (catp);
    \draw[edge] (cbp) -- (cbtp);
    \draw[imedge] (cabtp) -- (catp) -- (ctp);
    \draw[imedge] (cabtp) -- (cbtp) -- (ctp);

    \node[title] at (0,4.15) {$\UU(X^\Top)$};
    \node[upnode] (ux) at (uxp) {$X^\Top$};
    \node[upnode] (ua2) at (ua2p) {$\upar a$};
    \node[upnode] (ub2) at (ub2p) {$\upar b$};
    \node[upnode] (ut) at (utp) {$\upar\Top$};

    \node[title] at (6.1,4.15) {$\CC(X^\Top)$};
    \node[cnode] (cab) at (cabp) {$\widehat{\{a,b\}}$};
    \node[cnode] (ca) at (cap) {$\widehat{\{a\}}$};
    \node[cnode] (cb) at (cbp) {$\widehat{\{b\}}$};
    \node[imnode] (cabt) at (cabtp) {$\widehat{X^\Top}$};
    \node[imnode] (cat) at (catp) {$\widehat{\{a,\Top\}}$};
    \node[imnode] (cbt) at (cbtp) {$\widehat{\{b,\Top\}}$};
    \node[imnode] (ct) at (ctp) {$\widehat{\{\Top\}}$};

    \draw[-{Stealth[length=2.2mm]}, thick, draw=teal!65!black] (1.65,1.45) -- node[above, font=\scriptsize] {$e$} (4.05,2.25);
\end{tikzpicture}


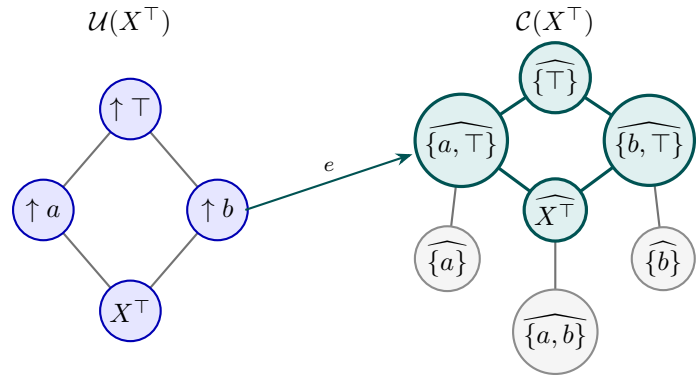
\captionof{figure}{After adjoining $\Top$, the image of $e(\upar F)=\widehat{F\cup\{\Top\}}$ is the highlighted diamond in $\CC(X^\Top)$.}
\end{minipage}
\end{center}

For example,
\[
        (r\circ e)(\upar\{a,b\})
        =r(\widehat{\{a,b,\Top\}})
        =\upar\{a,b,\Top\}
        =\upar\{a,b\},
\]
and the same computation works for every nonempty finite $F\subseteq X^\Top$.  Hence adjoining $\Top$ is precisely what makes the canonical retraction possible.
\end{ex}

We know that quasicontinuity is preserved by continuous retracts of dcpos \cite[Proposition~4.2.6]{GierzEtAl2003}.  The following lemma extends this result to directed spaces.

\begin{lema}\label{lema:qc-retract-directed}
Quasicontinuity is inherited by continuous retracts of directed spaces.
\end{lema}

\begin{proof}
Let $E$ be a retract of a quasicontinuous directed space $D$, with continuous maps $e\colon E\to D$ and $r\colon D\to E$ satisfying $r\circ e=\id_E$.  We use the local finite-neighbourhood characterization of quasicontinuity for directed spaces: for every point $x$ and open neighbourhood $U$ of $x$, there is a finite set $F\subseteq U$ such that $x\in\intr(\upar F)$.  Let $y\in U$, where $U$ is open in $E$.  Then $W=r^{-1}(U)$ is open in $D$, and $e(y)\in W$.  By quasicontinuity of $D$, there exists a finite $F\subseteq W$ such that
\[
        e(y)\in\intr_D(\upar F).
\]
Set $G=r(F)$.  Then $G\subseteq U$ is finite.  Moreover
\[
        V=e^{-1}\!\left(\intr_D(\upar F)\right)
\]
is an open neighbourhood of $y$ in $E$.  If $z\in V$, then $f\leq e(z)$ for some $f\in F$, hence $r(f)\leq z$.  Thus $V\subseteq\upar G$, so
\[
        y\in\intr_E(\upar G)
        \quad\text{and}\quad
        G\subseteq U.
\]
Since open sets are upper with respect to the specialization order, $\upar G\subseteq U$.
This establishes quasicontinuity of $E$.
\end{proof}

By combining the upper-powerdomain counterexample of \cite[Example~4.13]{ChenKouLyu2024} with the retract theorem, we obtain the following result.

\begin{theor}\label{the:main}
There exists a quasicontinuous dcpo $X$ such that $\PC(X)$ is not quasicontinuous.
\end{theor}

\begin{proof}
Let
\[
        X=2^{<\omega}\cup\{\Top\},
\]
where $2^{<\omega}$ denotes the prefix-ordered infinite binary tree and $\Top$ is a greatest element.  By the upper-powerdomain counterexample of \cite[Example~4.13]{ChenKouLyu2024}, $X$ is a quasicontinuous dcpo and the directed upper powerspace $\UU(\Sigma X)$ is not quasicontinuous.  Equivalently, by Theorem~\ref{the:qc-completion}, the upper powerdomain $\PU(X)$ is not quasicontinuous.

Since $X$ has a greatest element, Theorem~\ref{the:directed-retract} says that $\UU(\Sigma X)$ is a continuous retract of $\CC(\Sigma X)$ in $\Dtop$.  If $\CC(\Sigma X)$ were quasicontinuous, Lemma~\ref{lema:qc-retract-directed} would imply that $\UU(\Sigma X)$ is quasicontinuous, a contradiction.  Hence $\CC(\Sigma X)$ is not quasicontinuous.  Applying Theorem~\ref{the:qc-completion} once more, $\PC(X)$ is not quasicontinuous.
\end{proof}

\begin{corollar}
The convex powerdomain construction, equivalently the free dcpo-semilattice construction in $\DCPO$, does not preserve quasicontinuity.
\end{corollar}

\section{Preservation and Reflection Results}

We now collect the remaining preservation and reflection results for continuity and quasicontinuity.  By Theorem~\ref{the:qc-completion}, the directed powerspaces and their Scott-completed powerdomains have the same preservation and reflection behaviour for these two properties.  Thus the comparison table displayed in the Introduction applies simultaneously to the directed lower, upper and convex powerspaces and to the corresponding lower, upper and convex powerdomains.

\subsection{Well-known Results}

\begin{theor}[Lower powerdomains {\cite{AbramskyJung1994,GierzEtAl2003,BattenfeldSchroder2015,ChenKouXie2024RMJ}}]\label{the:lower-pres-ref}
Let $L$ be a dcpo.  The lower powerdomain $\PL(L)$ preserves and reflects both continuity and quasicontinuity:
\[
        L\text{ continuous}
        \Longleftrightarrow
        \PL(L)\text{ continuous},
        \qquad
        L\text{ quasicontinuous}
        \Longleftrightarrow
        \PL(L)\text{ quasicontinuous}.
\]
\end{theor}

\begin{theor}[Continuity preservation for directed powerspaces and free dcpo-algebras {\cite{AbramskyJung1994,GierzEtAl2003,Koslowski1997,JungMoshierVickers2008,Xie2021,ChenKouXie2024RMJ,ChenKouLyu2024Free}}]\label{the:continuity-preserved}
Let $X$ be a continuous directed space.  Then $\LL(X)$, $\UU(X)$ and $\CC(X)$ are continuous directed spaces.  At the dcpo level, if $L$ is a continuous dcpo, then every free dcpo-algebra over $L$, whenever it exists in the usual dcpo-algebraic setting, is continuous.  In particular,
\[
        L\text{ continuous}
        \Longrightarrow
        \PL(L),\ \PU(L),\ \PC(L)\text{ continuous}.
\]
\end{theor}

\subsection{Reflection of Quasicontinuity via Directed Powerspaces}

We first introduce two kinds of open subsets in directed powerspaces, and then use them to prove the reflection results for quasicontinuity and continuity. 

\begin{prop}[Upper Vietoris opens]\label{prop:box-open}
Let $X$ be a directed space and let $U\subseteq X$ be open.  Then
\[
        \Box U=\{\upar F\in U(X):\upar F\subseteq U\}
\]
is open in the directed upper powerspace $\UU(X)$.
\end{prop}

\begin{proof}
Suppose $\mathcal E\Rightarrow_U\upar F$ and $\upar F\in\Box U$.  Let $D_1,\ldots,D_k$ witness $\mathcal E\Rightarrow_U\upar F$.  Since $\upar F\subseteq U$, each point of $F$ belongs to $U$.  For each $i$, choose $a_i\in F\cap\Lim(D_i)$.  Since $U$ is open and $D_i\to a_i\in U$, choose $d_i\in D_i\cap U$.  By the condition of $\Rightarrow_U$, there is $\upar G\in\mathcal E$ such that
\[
        \upar G\subseteq\bigcup_{i=1}^k\upar d_i .
\]
Since $U$ is upper, $\bigcup_i\upar d_i\subseteq U$, and therefore $\upar G\in\Box U$.  Hence $\mathcal E\cap\Box U\neq\varnothing$.  It follows that $\Box U$ is open in $\UU(X)$.
\end{proof}

\begin{prop}[Lower Vietoris opens]\label{prop:diamond-open}
Let $X$ be a directed space and let $U\subseteq X$ be open.  Then
\[
        \Diamond U=\{\widehat F\in C(X):F\cap U\neq\varnothing\}
\]
is open in the directed convex powerspace $\CC(X)$.
\end{prop}

\begin{proof}
Suppose $\mathcal D\Rightarrow_C\widehat F$ and $\widehat F\in\Diamond U$.  Let $D_1,\ldots,D_k$ witness the convergence $\mathcal D\Rightarrow_C\widehat F$ in Definition~\ref{def:convex-convergence}.  Choose $a\in F\cap U$.  By the third clause, there is an index $i$ with $a\in\Lim(D_i)$.  Because $D_i\to a$ and $U$ is open, there is some $d\in D_i\cap U$.

By the lower-coordinate clause in the definition of $\Rightarrow_C$, the element $d\in D_i$ lies in
\[
        \bigcup\{\down G:\widehat G\in\mathcal D\}.
\]
Hence there exists $\widehat G\in\mathcal D$ and an element $g\in G$ such that $d\leq g$.  Since every open set is an upper set for the specialization order, $d\in U$ and $d\leq g$ imply $g\in U$.  Thus $g\in G\cap U$, so by the definition of $\Diamond U$ we have
\[
        \widehat G\in\Diamond U .
\]
As this same $\widehat G$ also belongs to $\mathcal D$, it follows that $\mathcal D\cap\Diamond U\neq\varnothing$. Hence $\Diamond U$ is open in $\CC(X)$.
\end{proof}

\begin{theor}[Upper powerspaces reflect quasicontinuity]\label{the:upper-qc-reflect}
Let $X$ be a directed space.  If $\UU(X)$ is quasicontinuous, then $X$ is quasicontinuous.  Consequently, if $\PU(L)$ is quasicontinuous, then $L$ is quasicontinuous.
\end{theor}

\begin{proof}
Let $x\in O$, where $O$ is open in $X$, and consider
\(
        \Box O=\{\upar F\in U(X):\upar F\subseteq O\}.
\)
By Proposition~\ref{prop:box-open}, $\Box O$ is open in $\UU(X)$, and the canonical map
\[
        i\colon X\to\UU(X),\qquad i(y)=\upar y,
\]
satisfies $i(x)=\upar x\in\Box O$. Assume that $\UU(X)$ is quasicontinuous.  Then there are finitely many elements
\(
        \upar F_1,\ldots,\upar F_m\in\Box O
\)
such that
\[
        \upar x\in
        \intr_{\UU}\bigl(\upar_U\{\upar F_1,\ldots,\upar F_m\}\bigr)
        \subseteq \Box O .
\]
Let $F=F_1\cup\cdots\cup F_m$.  Since each $\upar F_j\subseteq O$, we have $F\subseteq O$.  Pull back the displayed open neighbourhood along $i$:
\[
        W=i^{-1}\!\left(
        \intr_{\UU}\bigl(\upar_U\{\upar F_1,\ldots,\upar F_m\}\bigr)
        \right).
\]
Then by the continuity of $i$, $W$ is open in $X$ and $x\in W$.  If $y\in W$, then
\[
        \upar y\in\upar_U\{\upar F_1,\ldots,\upar F_m\},
\]
so $\upar F_j\leq_U\upar y$ for some $j$.  By the upper order this means
\(
        \upar y\subseteq\upar F_j,
\)
and hence $y\in\upar F_j\subseteq\upar F$.  Thus $W\subseteq\upar F$.  Since $F\subseteq O$ and $O$ is upper,
\[
        x\in\intr_X(\upar F)\subseteq O.
\]
This proves that $X$ is quasicontinuous.  The dcpo statement follows by applying the result to $X=\Sigma L$ and using Theorem~\ref{the:qc-completion}.
\end{proof}

\begin{theor}[Convex powerspaces reflect quasicontinuity]\label{the:convex-qc-reflect}
Let $X$ be a directed space.  If $\CC(X)$ is quasicontinuous, then $X$ is quasicontinuous.  Consequently, if $\PC(L)$ is quasicontinuous, then $L$ is quasicontinuous.
\end{theor}

\begin{proof}
Let $x\in O$, where $O$ is open in $X$, and put
\(
        \Diamond O=\{\widehat F\in C(X):F\cap O\neq\varnothing\}.
\)
By Proposition~\ref{prop:diamond-open}, $\Diamond O$ is open in $\CC(X)$.  Let $i$ be the canonical map
\[
        i\colon X\to\CC(X),\qquad i(y)=\widehat{\{y\}}.
\]
Then $i(x)=\widehat{\{x\}}\in\Diamond O$. Assume that $\CC(X)$ is quasicontinuous.  Then there exist finitely many elements
\(
        \widehat A_1,\ldots,\widehat A_m\in\Diamond O
\)
such that
\[
        \widehat{\{x\}}\in
        \intr_{\CC}\bigl(\upar_C\{\widehat A_1,\ldots,\widehat A_m\}\bigr)
        \subseteq\Diamond O .
\]
For each $j$ from $1$ to $m$, choose $a_j\in A_j\cap O$, and set $F=\{a_1,\ldots,a_m\}\subseteq O$.  Pulling back the displayed open neighbourhood along $i$, let
\[
        W=i^{-1}\!\left(
        \intr_{\CC}\bigl(\upar_C\{\widehat A_1,\ldots,\widehat A_m\}\bigr)
        \right).
\]
Then $W$ is open and $x\in W$.  If $y\in W$, then
\[
        \widehat{\{y\}}\in\upar_C\{\widehat A_1,\ldots,\widehat A_m\},
\]
so $\widehat A_j\leq_C\widehat{\{y\}}$ for some $j$.  The Egli--Milner order gives $\down A_j\subseteq\down y$, and hence $a_j\leq y$.  Thus $W\subseteq\upar F$.  Since $F\subseteq O$ and $O$ is upper,
\[
        x\in\intr_X(\upar F)\subseteq O.
\]
Therefore $X$ is quasicontinuous.  The dcpo statement follows by applying the result to $X=\Sigma L$ and using Theorem~\ref{the:qc-completion}.
\end{proof}

\subsection{Reflection of Continuity and A Separating Example}

The preceding counterexample shows that convex powerdomains do not preserve quasicontinuity.  We now add a complementary phenomenon: convex powerdomains are much more sensitive to continuity than upper powerdomains, because the lower component of a lens records genuine lower approximation in the base dcpo.

\begin{theor}[Continuity reflection]\label{the:continuity-reflection}
Let $X$ be a directed space.  If $\CC(X)$ is a continuous directed space, then $X$ is continuous.
\end{theor}

\begin{proof}
Let $x\in U$, where $U$ is open in $X$.  By Proposition~\ref{prop:diamond-open}, $\Diamond U$ is an open neighbourhood of $\widehat{\{x\}}$ in $\CC(X)$.  Since $\CC(X)$ is continuous, there exists a finite nonempty $A\subseteq X$ such that
\[
        \widehat{\{x\}}\in\intr(\upar_C\widehat A)
        \subseteq \Diamond U .
\]
Because $\widehat A\in\Diamond U$, choose $a\in A\cap U$.  Put
\[
        W=i_X^{-1}\bigl(\intr(\upar_C\widehat A)\bigr),
        \qquad
        i_X(y)=\widehat{\{y\}} .
\]
Then $W$ is open in $X$ and $x\in W$.  If $y\in W$, then $\widehat A\leq_C\widehat{\{y\}}$, hence $\down A\subseteq\down y$.  In particular, $a\leq y$, so $W\subseteq\upar a$.  Therefore
\[
        x\in\intr_X(\upar a)\subseteq U .
\]
By Definition~\ref{def:directed-continuity}, $X$ is continuous.
\end{proof}

\begin{corollar}\label{cor:convex-continuity-equivalence}
For any dcpo $L$,
        $\PC(L)\text{ is continuous}$
if and only if
        $L\text{ is continuous}$.
\end{corollar}

\begin{proof}
The forward implication is Theorem~\ref{the:continuity-reflection}, together with the Scott-completion representation.  The reverse implication is the known continuity-preservation result of Theorem~\ref{the:continuity-preserved} applied to $L$.
\end{proof}

\begin{note}
For upper powerdomains, continuity reflection is false.  For example, the two-chain dcpo of Example~\ref{ex:two-chain} is quasicontinuous but not continuous, yet its upper powerdomain is continuous \cite[Example~4.13]{ChenKouLyu2024}. 
\end{note}

We next display an example showing that quasicontinuity of a convex powerdomain need not imply continuity.

\begin{ex}[The two-chain dcpo]\label{ex:two-chain}
Let
\(
        Y=\{a_n:n\geq 1\}\cup\{b_n:n\geq 1\}\cup\{\omega\},
\)
ordered by
\[
        a_1<a_2<\cdots<\omega,\qquad
        b_1<b_2<\cdots<\omega,
\]
with the two chains incomparable below $\omega$.  Then $Y$ is a dcpo, with $\bigvee_n a_n=\bigvee_n b_n=\omega$. It is easy to see that $Y$ is quasicontinuous but not continuous.  For $1\leq p\leq m$, write
\[
        A[p,m]=\{a_i:p\leq i\leq m\},\qquad
        B[p,m]=\{b_i:p\leq i\leq m\},
\]
and
\[
        A[p,\infty]=\{a_i:i\geq p\}\cup\{\omega\},\qquad
        B[p,\infty]=\{b_i:i\geq p\}\cup\{\omega\}.
\]
\end{ex}

We now show, at the directed-space level, that its convex powerdomain is nevertheless quasicontinuous.

\begin{lema}[Scott description of the two-chain convex carrier]\label{lema:two-chain-scott-carrier}
The ordered set $(C(\Sigma Y),\leq_C)$ is a dcpo and
\[
        \mathcal{O}_{\Rightarrow_C}(C(\Sigma Y))=\sig(C(\Sigma Y),\leq_C).
\]
\end{lema}

\begin{proof}
Identify a finite lens $\widehat F$ with the subset $\down F\cap\upar F$ it determines; this gives a one-to-one correspondence.  Then the elements of $C(\Sigma Y)$ can be characterized exactly as the following finite-type lenses:
\[
        A[p,m],\quad B[q,n],\quad A[p,m]\cup B[q,n],
\]
together with
\[
        A[p,\infty],\quad B[q,\infty],\quad
        A[p,\infty]\cup B[q,\infty],\quad \{\omega\},\quad Y .
\]
We claim that $C(\Sigma Y)$ is closed under directed suprema.  For a directed family, the finite endpoints on each branch either stabilize at a finite value or move cofinally to $\omega$; taking the Scott closure of the union of the lower components and the intersection of the upper components gives precisely one of the listed lens types.  This candidate is the least Egli--Milner upper bound, so $(C(\Sigma Y),\leq_C)$ is a dcpo.

We next compare the two topologies.  Since the Scott topology is the coarsest directed topology on any dcpo, we have $\sig(C(\Sigma Y),\leq_C)\subseteq\mathcal{O}_{\Rightarrow_C}(C(\Sigma Y))$. We need only show that every directed open set $\mathcal U\in\mathcal{O}_{\Rightarrow_C}(C(\Sigma Y))$ is Scott open.  It is enough to show that for every directed family $\mathcal D$ of $C(\Sigma Y)$, $\mathcal D \Rightarrow_C \bigvee\mathcal D$.
Let $\mathcal D$ be directed and let $\widehat F=\bigvee\mathcal D$.  
If the supremum lens $\widehat F$ is finite-type, the endpoint description above implies that it is already reached by a member of the family, and singleton witnesses suffice.  If the supremum has lower component $Y$, use singleton witnesses for the finite upper thresholds which are eventually reached, and use the cofinal branch witnesses
\[
        A=\{a_i:i\geq 1\},\qquad B=\{b_i:i\geq 1\}
\]
for the occurrences of $\omega$ forced by cofinal movement along the two branches.  The lower and limit-covering clauses are immediate; the fourth clause follows because any chosen tuple imposes only finitely many endpoint requirements, which directedness realizes in one member of the family. Therefore $\mathcal D\Rightarrow_C\widehat F$.  This proves that $\mathcal{O}_{\Rightarrow_C}(C(\Sigma Y))\subseteq\sig(C(\Sigma Y),\leq_C)$, and hence the two topologies coincide.
\end{proof}

\begin{lema}[Quasicontinuity of the two-chain convex powerspace]\label{lema:two-chain-convex-qc}
The directed convex powerspace $\CC(\Sigma Y)$ is quasicontinuous.  Consequently, $\PC(Y)$ is quasicontinuous.
\end{lema}

\begin{proof}
By Lemma~\ref{lema:two-chain-scott-carrier}, $\CC(\Sigma Y)$ is the Scott space of the dcpo $(C(\Sigma Y),\leq_C)$.  It is enough to prove that this dcpo is quasicontinuous.  

Call a lens finite if it is one of the finite interval types
\[
        A[p,m],\qquad B[q,n],\qquad A[p,m]\cup B[q,n].
\]
For such a finite lens $L$ and $r,s\in\mathbb N$, put
\[
        W(L,r,s)=
        \upar_C L
        \cup\upar_C\widehat{\{a_r\}}
        \cup\upar_C\widehat{\{b_s\}} .
\]
This set is Scott open.  It is clearly an upper set.  Let $\mathcal D$ be directed and suppose $\bigvee\mathcal D\in W(L,r,s)$.  If the finitely many endpoint inequalities which put $\bigvee\mathcal D$ into one of the three principal upsets are already realized by some member of $\mathcal D$, then $\mathcal D\cap W(L,r,s)\neq\varnothing$.  Otherwise, by the endpoint description in Lemma~\ref{lema:two-chain-scott-carrier}, the only possible obstruction is that an endpoint reaches $\omega$ by moving cofinally along at least one branch.  Cofinal movement along the $a$-branch gives a member of $\mathcal D$ above $\widehat{\{a_r\}}$; cofinal movement along the $b$-branch gives a member above $\widehat{\{b_s\}}$.  Hence $\mathcal D$ always meets $W(L,r,s)$, and $W(L,r,s)$ is Scott open. 

Now let $\mathcal U$ be Scott open in $(C(\Sigma Y),\leq_C)$ and let $K\in\mathcal U$.  The element $\widehat{\{\omega\}}$ is the greatest element of $C(\Sigma Y)$, so $\widehat{\{\omega\}}\in\mathcal U$.  The two directed chains
\[
        \widehat{\{a_1\}}\leq_C\widehat{\{a_2\}}\leq_C\cdots,
        \qquad
        \widehat{\{b_1\}}\leq_C\widehat{\{b_2\}}\leq_C\cdots
\]
both have supremum $\widehat{\{\omega\}}$.  Scott openness gives $r,s\in\mathbb N$ such that
\[
        \widehat{\{a_r\}}\in\mathcal U,\qquad
        \widehat{\{b_s\}}\in\mathcal U .
\]

Choose a finite lens $L\in\mathcal U$ with $L\leq_C K$: if $K$ is finite, take $L=K$; if $K$ is a tail lens, take a sufficiently long finite interval in the directed chain whose supremum is $K$, and the Scott openness of $\mathcal U$ provides such an interval; if $K=\widehat{\{\omega\}}$, take $L=\widehat{\{a_r\}}$; if $K=Y$, use the directed finite two-branch intervals $A[1,m]\cup B[1,n]$, whose supremum is $Y$, and choose one lying in $\mathcal U$.

Set
\[
        \mathcal S=\{L,\widehat{\{a_r\}},\widehat{\{b_s\}}\}.
\]
Then $\mathcal S$ is finite, $\mathcal S\subseteq\mathcal U$, and $K\in\upar_C\mathcal S$.  Moreover,
\(
        \upar_C\mathcal S=W(L,r,s)
\)
is Scott open, and therefore is also open in $\CC(\Sigma Y)$ by Lemma~\ref{lema:two-chain-scott-carrier}.  Thus
\(
        K\in\intr(\upar_C\mathcal S)\subseteq\upar_C\mathcal S\subseteq\mathcal U .
\)
This proves that $\CC(\Sigma Y)$ is quasicontinuous.  The statement for $\PC(Y)$ follows from Theorem~\ref{the:qc-completion}.
\end{proof}

Therefore, for quasicontinuous dcpos $L$,
\[
        \PC(L)\text{ quasicontinuous}
        \quad\not\Rightarrow\quad
        \PC(L)\text{ continuous}.
\]
This is different from the upper construction~\cite{ChenKouLyu2024}.  For a quasicontinuous dcpo $L$,
\[
        \UU(\Sigma L)\text{ continuous}
        \quad\Longleftrightarrow\quad
        \text{the topology of }\UU(\Sigma L)\text{ is }\sig(U(\Sigma L))
        \quad\Longleftrightarrow\quad
        \PU(L)\cong \mathcal{Q}(L) .
\]
For convex powerdomains, no analogous equivalence chain is available.  The continuity theorem gives the sufficient implication
\[
        L\text{ continuous}
        \quad\Longleftrightarrow\quad
        \PC(L)\text{ continuous}
        \quad\Longrightarrow\quad
        \mathcal{O}_{\Rightarrow_C}(C(\Sigma L))=\sig(C(\Sigma L),\leq_C),
\]
which is proved in the next proposition.  The two-chain example above shows that an abstract compact-lens isomorphism
\(
        \PC(L)\cong \KConv(L)
\) and the equality of the directed convex powerspace topology with the Scott topology 
do not force continuity of $\CC(\Sigma L)$ or $\PC(L)$.  

\begin{prop}[Scott topology for continuous bases]\label{prop:convex-scott-topology-continuous}
Let $L$ be a continuous dcpo.  Then the topology of the directed convex powerspace $\CC(\Sigma L)$ coincides with the Scott topology of its specialization order.
\end{prop}

\begin{proof}
Write $C_L=C(\Sigma L)$.  For nonempty finite subsets $A,F\subseteq L$, write
\(
        \widehat A\prec_C\widehat F
\)
if each $a\in A$ is way-below some $f\in F$, and each $f\in F$ is above some $a\in A$ with $a\ll f$.  We first recall the finite Plotkin approximation fact needed below.  For fixed $\widehat F\in C_L$, the set
\[
        B_F=\{\widehat A\in C_L:\widehat A\prec_C\widehat F\}
\]
is directed and has supremum $\widehat F$ in $C_L$.

Moreover $\widehat A\prec_C\widehat F$ implies $\widehat A\ll\widehat F$ in the ordered set $C_L$.

Indeed, if $\widehat A,\widehat B\in B_F$, then for each $f\in F$ collect the finitely many elements of $A\cup B$ assigned below $f$, and also one witness from $A$ and one from $B$ way-below $f$.  Since $\{x:x\ll f\}$ is directed, choose $c_f\ll f$ above this finite set and put $C=\{c_f:f\in F\}$.  Then $\widehat A,\widehat B\leq_C\widehat C\prec_C\widehat F$.  Thus $B_F$ is directed.  It is clear that every member of $B_F$ is below $\widehat F$.  If $\widehat G$ is any upper bound of $B_F$, then for each $f\in F$ and each $b\ll f$, choosing a finite approximant $\widehat A_b\prec_C\widehat F$ with $b\in A_b$ gives $b\in\down G$; by continuity, $f\in\down G$.  Hence $\down F\subseteq\down G$.  If some $g\in G$ were above no element of $F$, choose $b_f\ll f$ with $b_f\not\leq g$ for every $f\in F$; then $\widehat{\{b_f:f\in F\}}\prec_C\widehat F\leq_C\widehat G$, forcing $g\in\uparrow\{b_f:f\in F\}$, a contradiction.  Thus $\upar G\subseteq\upar F$, and $\widehat F\leq_C\widehat G$.  Hence $\sup B_F=\widehat F$.  The way-below assertion follows from the same finite interpolation argument: if $\widehat A\prec_C\widehat F\leq_C\sup\mathcal E$ for a directed family $\mathcal E$ with existing supremum, interpolate to a finite $\widehat C$ with $\widehat A\leq_C\widehat C\prec_C\sup\mathcal E$, and the preceding finite argument gives some $\widehat E\in\mathcal E$ with $\widehat C\leq_C\widehat E$.

First, we prove that every Scott-open subset of $C_L$ is $\Rightarrow_C$-open.  Let $\mathcal V\in\sig(C_L,\leq_C)$ and suppose $\mathcal D\Rightarrow_C\widehat F\in\mathcal V$, witnessed by directed sets $D_1,\ldots,D_k$ in $\Sigma L$.  By the preceding paragraph, choose $\widehat A\prec_C\widehat F$ with $\widehat A\in\mathcal V$; since $\mathcal V$ is an upper set, it is enough to find $\widehat G\in\mathcal D$ with $\widehat A\leq_C\widehat G$.

For each $a\in A$, choose $f_a\in F$ with $a\ll f_a$.  By the limit-covering clause for $\Rightarrow_C$, choose $i(a)$ with $f_a\in\Lim(D_{i(a)})$, and then choose $d_a\in D_{i(a)}$ with $a\leq d_a$.  The lower-coordinate clause gives $\widehat G_a\in\mathcal D$ with $d_a\in\down G_a$.  Directedness of $\mathcal D$ gives $\widehat G_0\in\mathcal D$ above all $\widehat G_a$, hence $\down A\subseteq\down G_0$.

For the upper component, choose $f_i\in F\cap\Lim(D_i)$ for each $i$, then choose $a_i\in A$ with $a_i\ll f_i$ and $d_i\in D_i$ with $a_i\leq d_i$.  The upper-coordinate clause yields $\widehat G_1\in\mathcal D$ such that
\[
        \upar G_1\subseteq\bigcup_i\upar d_i
        \subseteq\bigcup_i\upar a_i
        \subseteq\upar A .
\]
Taking $\widehat G\in\mathcal D$ above $\widehat G_0$ and $\widehat G_1$, we obtain $\down A\subseteq\down G$ and $\upar G\subseteq\upar A$, hence $\widehat A\leq_C\widehat G$.  Therefore $\widehat G\in\mathcal V$, proving $\mathcal D\cap\mathcal V\neq\varnothing$.

Conversely, let $\mathcal V\in\mathcal{O}_{\Rightarrow_C}(C_L)$.  It is an upper set.
Let $\mathcal D\subseteq C_L$ be directed with existing supremum $\widehat F=\sup\mathcal D\in\mathcal V$, say $F=\{f_1,\ldots,f_n\}$.  For each $i$, put
\[
        E_i=\{e\in L:e\ll f_i
        \text{ and }e\in\bigcup_{\widehat G\in\mathcal D}\down G\}.
\]
If $e\ll f_i$, choose a finite set $A_e$ with $e\in A_e$ and $\widehat A_e\prec_C\widehat F$.  Since $\widehat A_e\ll\widehat F=\sup\mathcal D$, there is $\widehat G\in\mathcal D$ with $\widehat A_e\leq_C\widehat G$, so $e\in\down G$.  Thus $E_i=\{e:e\ll f_i\}$, and each $E_i$ is directed and converges to $f_i$ in $\Sigma L$.  These $E_i$ satisfy the lower-coordinate and limit-covering clauses of $\Rightarrow_C$.

For a tuple $(e_1,\ldots,e_n)\in E_1\times\cdots\times E_n$, let $A=\{e_1,\ldots,e_n\}$.  Then $\widehat A\prec_C\widehat F$, hence $\widehat A\ll\widehat F=\sup\mathcal D$.  Choose $\widehat G\in\mathcal D$ with $\widehat A\leq_C\widehat G$.  Then
\(
        \upar G\subseteq\upar A=\bigcup_{i=1}^n\upar e_i,
\)
which is the upper-coordinate clause.  Therefore $\mathcal D\Rightarrow_C\widehat F$.  Since $\mathcal V$ is $\Rightarrow_C$-open and $\widehat F\in\mathcal V$, we have $\mathcal D\cap\mathcal V\neq\varnothing$.  Thus $\mathcal V$ is Scott open. Hence $\mathcal{O}_{\Rightarrow_C}(C_L)\subseteq\sig(C_L,\leq_C)$, and the two topologies coincide.
\end{proof}

\begin{note}
The failure of the implication from a compact-lens representation to continuity is witnessed by the two-chain dcpo $Y$.  For $Y$, the lens poset $\KConv(Y)$ is very small.  Its Scott closed lower sets are exactly the finite two-sided prefixes
\[
        \{a_1,\ldots,a_m\}\cup\{b_1,\ldots,b_n\}
\]
with one side allowed to be empty, together with $Y$ itself.  Indeed, an infinite prefix on either chain has supremum $\omega$, and a lower set containing $\omega$ is all of $Y$.  Its nonempty compact saturated subsets are the finite generated upper sets
\[
        \upar F\qquad(\varnothing\neq F\subseteq_{\fin}Y),
\]
including $\{\omega\}$, $\upar a_m$, $\upar b_n$, $\upar\{a_m,b_n\}$ and $Y$.  Hence every nonempty lens $C\cap Q$ is finite generated, and we have
\[
      \CC(\Sigma Y) = \Sigma \PC(Y) \text{ and  } \PC(Y)  \cong \KConv(Y).
\]
Nevertheless, $\CC(\Sigma Y)$ is not continuous by Theorem~\ref{the:continuity-reflection}, since $Y$ is not continuous.  Equivalently, $\PC(Y)$ is not continuous by Corollary~\ref{cor:convex-continuity-equivalence}.  Therefore a bare order isomorphism with a compact-lens poset does not imply continuity of the directed convex powerspace.
\end{note}

The entries of the comparison table in the Introduction are justified as follows.  The equivalence between directed-powerspace and powerdomain entries follows from the Scott-completion transfer in Theorem~\ref{the:qc-completion}.  For the lower construction, all four entries are the known Hoare-powerdomain results in Theorem~\ref{the:lower-pres-ref}.  Continuity preservation for all three powerdomains follows from Theorem~\ref{the:continuity-preserved}; for lower and upper powerspaces in directed spaces see also \cite{ChenKouXie2024RMJ}.  Upper powerdomains do not reflect continuity by the standard two-chain counterexample; see \cite{HeckmannKeimel2013,ChenKouLyu2024}.  Convex powerdomains reflect continuity by Corollary~\ref{cor:convex-continuity-equivalence}.  For quasicontinuity, upper powerdomains do not preserve it by \cite{ChenKouLyu2024}, while Theorem~\ref{the:upper-qc-reflect} gives reflection.  Convex powerdomains do not preserve quasicontinuity by Theorem~\ref{the:main}, and Theorem~\ref{the:convex-qc-reflect} gives reflection.

\section{A Final Retract Analysis}

For completeness, and also to clarify the role of the directed-space framework, we collect here two ways to obtain a retraction from the convex powerdomain onto the upper powerdomain when the base dcpo has a greatest element.  The first is obtained formally by Scott completion from the directed-space retraction.  The second is a direct construction at the level of free dcpo-algebras.

\begin{theor}[Completion-induced dcpo retraction]\label{the:completion-dcpo-retract}
Let $L$ be a dcpo with greatest element.  Then $\PU(L)$ is a Scott-continuous retract of $\PC(L)$.
\end{theor}

\begin{proof}
Apply Theorem~\ref{the:directed-retract} to the Scott space $\Sigma L$.  By Theorem~\ref{the:scott-completion}, the Scott completion of the directed upper powerspace is $\PU(L)$, while the Scott completion of the directed convex powerspace is the free dcpo-semilattice $\PC(L)$.  By the universal property of Scott completion, the continuous maps
\[
        e\colon\UU(\Sigma L)\to\CC(\Sigma L),
        \qquad
        r\colon\CC(\Sigma L)\to\UU(\Sigma L)
\]
uniquely extend to Scott-continuous maps
\[
        \bar e\colon \PU(L)\to\PC(L),
        \qquad
        \bar r\colon \PC(L)\to\PU(L).
\]
Since $r\circ e=\id_{\UU(\Sigma L)}$, uniqueness of the extension gives $\bar r\circ\bar e=\id_{\PU(L)}$.  Hence $\PU(L)$ is a Scott-continuous retract of $\PC(L)$.
\end{proof}

\begin{lema}[Generation by the unit]\label{lema:generation-dcpo}
Let $(A,\eta)$ be a free dcpo-algebra over a dcpo $L$ in a category of dcpo-algebras whose morphisms are Scott-continuous homomorphisms.  If $B\subseteq A$ is a sub-dcpo and a subalgebra such that $\eta[L]\subseteq B$, then $B=A$.
\end{lema}

\begin{proof}
View $\eta$ as a map $\eta_B\colon L\to B$.  By the freeness of $A$, there exists a unique homomorphism
\[
        h\colon A\to B
\]
such that $h\circ\eta=\eta_B$.  Let $i\colon B\hookrightarrow A$ be the inclusion.  Then $i\circ h\colon A\to A$ is a homomorphism and
\[
        (i\circ h)\circ\eta=i\circ\eta_B=\eta.
\]
The identity map $\id_A$ has the same property.  By uniqueness in the universal property of $A$,
\[
        i\circ h=\id_A.
\]
Thus $i$ is surjective, and hence $B=A$.
\end{proof}

\begin{theor}[Direct dcpo retraction]\label{the:direct-dcpo-retract}
Let $L$ be a dcpo with greatest element $\Top$.  Then $\PU(L)$ is a Scott-continuous retract of $\PC(L)$.
\end{theor}

\begin{proof}
Let
\[
        A=\PC(L),\qquad \eta=\eta_C\colon L\to A,
        \qquad c=\eta(\Top).
\]
Since $\eta$ is monotone, $\eta(x)\leq c$ for every $x\in L$.  Consider
\[
        B=\down c=\{a\in A:a\leq c\}.
\]
The set $B$ is a sub-dcpo of $A$: if $D\subseteq B$ is directed, then $\bigvee_A D\leq c$.  It is also closed under the semilattice operation, because if $a,b\leq c$, then by monotonicity and idempotency,
\[
        a+b\leq c+c=c.
\]
Moreover $\eta[L]\subseteq B$.  By Lemma~\ref{lema:generation-dcpo}, $B=A$.  Hence $c$ is the greatest element of $A=\PC(L)$.

Define
\[
        \rho\colon A\to A,\qquad \rho(a)=a+c.
\]
The map $\rho$ is Scott continuous.  By associativity and idempotency,
\[
        \rho(\rho(a))=(a+c)+c=a+c=\rho(a),
\]
so $\rho$ is idempotent.  Let
\[
        T=\rho[A]=\{a+c:a\in A\}.
\]
Then $T$ is a Scott-continuous retract of $A$, hence a dcpo.  It is closed under $+$: if $a,b\in T$, then
\[
        \rho(a+b)=a+b+c=(a+c)+b=a+b,
\]
so $a+b\in T$.

We claim that $T$ is deflationary.  Let $a,b\in T$.  Since $c$ is greatest in $A$, $b\leq c$.  Therefore
\[
        a+b\leq a+c=\rho(a)=a.
\]
Thus $T$ is a deflationary dcpo-semilattice.

Define
\[
        j\colon L\to T,\qquad j(x)=\eta_C(x)+c.
\]
By the universal property of $\PU(L)$, there exists a unique deflationary homomorphism
\[
        \alpha\colon \PU(L)\to T
        \quad\text{such that}\quad
        \alpha\circ\eta_U=j.
\]
By the universal property of $\PC(L)$, since $\PU(L)$ is a dcpo-semilattice, there exists a unique semilattice homomorphism
\[
        \beta\colon \PC(L)\to\PU(L)
        \quad\text{such that}\quad
        \beta\circ\eta_C=\eta_U.
\]
Let $i\colon T\hookrightarrow\PC(L)$ be the inclusion and set
\[
        e=i\circ\alpha\colon \PU(L)\to\PC(L),
        \qquad
        r=\beta\colon \PC(L)\to\PU(L).
\]
Both maps are Scott continuous.  For every $x\in L$,
\[
        (r\circ e)(\eta_U(x))
        =\beta(\eta_C(x)+c)
        =\eta_U(x)+\eta_U(\Top).
\]
Since $x\leq\Top$, we have $\eta_U(x)\leq\eta_U(\Top)$.  In a deflationary dcpo-semilattice the operation is meet, hence
\[
        \eta_U(x)+\eta_U(\Top)=\eta_U(x).
\]
Thus $r\circ e$ and $\id_{\PU(L)}$ agree on the generators $\eta_U[L]$.  By uniqueness in the universal property of $\PU(L)$,
\[
        r\circ e=\id_{\PU(L)}.
\]
Therefore $\PU(L)$ is a Scott-continuous retract of $\PC(L)$.
\end{proof}

\begin{note}
The direct proof above is illuminating, but it also reveals why the directed powerspace framework is conceptually cleaner.  At the dcpo level one must first employ the generation lemma to prove that $\eta_C(\Top)$ is actually the greatest element of $\PC(L)$, then construct the idempotent $\rho(a)=a+\eta_C(\Top)$, identify its image as a deflationary dcpo-semilattice, and finally utilize two universal properties to manufacture the maps $e$ and $r$.

In the directed powerspace proof, the same retraction is directly visible on finite generators:
\[
        e(\upar F)=\widehat{F\cup\{\Top\}},
        \qquad
        r(\widehat F)=\upar F.
\]
Continuity is verified via the convergence definitions $\Rightarrow_U$ and $\Rightarrow_C$, and the identity $r\circ e=\id_{\UU(X)}$ follows immediately from $\upar(F\cup\{\Top\})=\upar F$.  Thus the directed-space construction separates the combinatorial content of the retraction from the completion step.  This is precisely the advantage exploited in the passage from the upper-powerdomain counterexample to the convex-powerdomain counterexample.
\end{note}


\begin{thebibliography}{99}

\bibitem{AbramskyJung1994}
S.~Abramsky and A.~Jung,
\newblock Domain theory,
\newblock in S.~Abramsky, D.~M. Gabbay and T.~S.~E. Maibaum (eds.),
\emph{Handbook of Logic in Computer Science}, Vol.~3, Oxford University Press, 1994, pp.~1--168.

\bibitem{BattenfeldSchroder2015}
I.~Battenfeld and M.~Schröder,
\newblock Observationally-induced lower and upper powerspace constructions,
\newblock \emph{Journal of Logical and Algebraic Methods in Programming} \textbf{84}(5) (2015), 668--682.

\bibitem{ChenKouLyu2024}
Y.~Chen, H.~Kou and Z.~Lyu,
\newblock Upper powerdomains of quasicontinuous dcpos,
\newblock \emph{Theoretical Computer Science} \textbf{1006} (2024), 114663.

\bibitem{ChenKouLyu2024Free}
Y.~Chen, H.~Kou and Z.~Lyu,
\newblock Free algebras over continuous spaces,
\newblock \emph{Houston Journal of Mathematics} \textbf{50}(4) (2024), 751--776.

\bibitem{ChenKouLyuXie2024Cones}
Y.~Chen, H.~Kou, Z.~Lyu and X.~Xie,
\newblock A construction of free dcpo-cones,
\newblock \emph{Mathematical Structures in Computer Science} \textbf{34}(1) (2024), 63--79.

\bibitem{ChenKouXie2024RMJ}
Y.~Chen, H.~Kou and X.~Xie,
\newblock Upper and lower powerspaces of directed spaces,
\newblock \emph{Rocky Mountain Journal of Mathematics} \textbf{54}(5) (2024), 1299--1314.

\bibitem{Erne2009}
M.~Erné,
\newblock Infinite distributive laws versus local connectedness and compactness properties,
\newblock \emph{Topology and its Applications} \textbf{156}(12) (2009), 2054--2069.

\bibitem{GierzLawsonStralka1983}
G.~Gierz, J.~D. Lawson and A.~R. Stralka,
\newblock Quasicontinuous posets,
\newblock \emph{Houston Journal of Mathematics} \textbf{9}(2) (1983), 191--208.

\bibitem{GierzEtAl2003}
G.~Gierz, K.~H. Hofmann, K.~Keimel, J.~D. Lawson, M.~Mislove and D.~S. Scott,
\newblock \emph{Continuous Lattices and Domains},
\newblock Encyclopedia of Mathematics and its Applications, Vol.~93, Cambridge University Press, 2003.

\bibitem{Heckmann1991}
R.~Heckmann,
\newblock Power domain constructions,
\newblock \emph{Science of Computer Programming} \textbf{17}(1--3) (1991), 77--117.

\bibitem{Heckmann1991IPL}
R.~Heckmann,
\newblock Lower and upper power domain constructions commute on all cpos,
\newblock \emph{Information Processing Letters} \textbf{40}(1) (1991), 7--11.

\bibitem{Heckmann2001}
R.~Heckmann,
\newblock Characterising FS domains by means of power domains,
\newblock \emph{Theoretical Computer Science} \textbf{264}(2) (2001), 195--203.

\bibitem{HeckmannKeimel2013}
R.~Heckmann and K.~Keimel,
\newblock Quasicontinuous domains and the Smyth powerdomain,
\newblock \emph{Electronic Notes in Theoretical Computer Science} \textbf{298} (2013), 215--232.

\bibitem{JungMoshierVickers2008}
A.~Jung, M.~A. Moshier and S.~Vickers,
\newblock Presenting dcpos and dcpo algebras,
\newblock in A.~Bauer and M.~Mislove (eds.),
\emph{Proceedings of the 24th Conference on the Mathematical Foundations of Programming Semantics (MFPS XXIV)},
\newblock \emph{Electronic Notes in Theoretical Computer Science} \textbf{218} (2008), 209--229.

\bibitem{KeimelLawson2009}
K.~Keimel and J.~D. Lawson,
\newblock $D$-completions and the $d$-topology,
\newblock \emph{Annals of Pure and Applied Logic} \textbf{159}(3) (2009), 292--306.

\bibitem{KeimelLawson2009Ops}
K.~Keimel and J.~D. Lawson,
\newblock Extending algebraic operations to $D$-completions,
\newblock \emph{Electronic Notes in Theoretical Computer Science} \textbf{249} (2009), 93--116.

\bibitem{Koslowski1997}
J.~Koslowski,
\newblock Note on free algebras over continuous domains,
\newblock \emph{Theoretical Computer Science} \textbf{179}(1--2) (1997), 421--425.

\bibitem{GoubaultLarrecq2013}
J.~Goubault-Larrecq,
\newblock \emph{Non-Hausdorff Topology and Domain Theory: Selected Topics in Point-Set Topology},
\newblock New Mathematical Monographs, Vol.~22, Cambridge University Press, 2013.

\bibitem{LuoXu2017}
S.~Luo and X.~Xu,
\newblock On monotone determined spaces,
\newblock \emph{Electronic Notes in Theoretical Computer Science} \textbf{333} (2017), 63--72.

\bibitem{Plotkin1976}
G.~D. Plotkin,
\newblock A powerdomain construction,
\newblock \emph{SIAM Journal on Computing} \textbf{5}(3) (1976), 452--487.

\bibitem{Smyth1978}
M.~B. Smyth,
\newblock Power domains,
\newblock \emph{Journal of Computer and System Sciences} \textbf{16}(1) (1978), 23--36.

\bibitem{Xie2021}
X.~Xie,
\newblock \emph{The Powerstructures of Directed Spaces and Related Problems},
\newblock Ph.D. thesis, Sichuan University, 2021.

\end{thebibliography}
\end{document}